\documentclass[12pt]{amsart}
\usepackage{amsfonts,amsmath,amssymb,amscd}
\usepackage{enumerate}
\usepackage{enumitem}
\usepackage{comment}
\newtheorem{theorem}{Theorem}[section]
\newtheorem{lemma}[theorem]{Lemma}

\newtheorem{remark}[theorem]{Remark}
\newtheorem{proposition}[theorem]{Proposition}

\newcommand{\Z}{\mbox{$\mathbb Z$}}

\newcommand{\N}{\mbox{$\mathbb N$}}     
\newcommand{\C}{\mbox{$\mathbb C$}}     

\begin{document}
	\title[Zeros and $S$-units  in recurrence sequences]{Zeros and $S$-units in sums of terms of recurrence sequences in function fields}
	
	\author[Darsana]{Darsana N}
	\address{Darsana N, Department of Mathematics, National Institute of Technology Calicut, 
		Kozhikode-673 601, India.}
	\email{darsana\_p230059ma@nitc.ac.in; darsanasrinarayan@gmail.com}

	\author[Rout]{S. S. Rout}
	\address{Sudhansu Sekhar Rout, Department of Mathematics, National Institute of Technology Calicut, 
		Kozhikode-673 601, India.}
	\email{sudhansu@nitc.ac.in; lbs.sudhansu@gmail.com}

	\dedicatory{}
	\begin{abstract}
		Let $(U_n)_{n\geq 0}$ be a non-degenerate linear recurrence sequence with order at least two defined over a function field and $\mathcal{O}_S^*$ be the set of $S$-units. In this paper, we use a result of Brownawell and Masser  to prove effective results related to the Diophantine equations concerning linear recurrence sequences and $S$-units.  In particular, we provide a finiteness result for the solutions of the Diophantine equation $U_{n_1} + \cdots + U_{n_r} \in \mathcal{O}_S^*$ in nonnegative integers $n_1, \ldots, n_r$. Furthermore, we study the finiteness result of the Diophantine equation $U_n+V_m+W_\ell = 0$ in $(n, m, \ell)\in \N^3$, where $U_n,V_m,W_\ell$ are simple linear recurrence sequences in the function field. 
	\end{abstract}
	\maketitle
	\pagenumbering{arabic}
	\pagestyle{headings}

	\section{Introduction}
	Let $k$ be an algebraically closed field of zero characteristic  and $K$ be a function field in one variable over $k$.  For a given positive integer  $d$, let $(U_{n})_{n \geq 0}$ be a linear recurrence sequence  of order $d$ defined by
	\begin{equation}\label{eq4}
		U_{n} = w_1U_{n-1} + \dots +w_dU_{n-d},
	\end{equation}
	where $w_1,\dots, w_d , U_0,\dots,U_{d-1}\in K$. The characteristic polynomial of the sequence $(U_n)_{n\geq 0}$ is given by
	\begin{equation}\label{eq5}
		f(X):= X^d - w_1X^{d-1}-\dots-w_d \in K[X].
	\end{equation}
	Let $\alpha_1,\dots,\alpha_t$ be distinct roots of the characteristic polynomial $f(X)$ in the splitting field $L$ of $f(X)$. So $L$ is a finite algebraic extension of $K$ and therefore a function field over $k$ of genus $\mathfrak{g}$ (say).  It is well known that for $n\geq 0$,
	\begin{equation}\label{eq6}
		U_n=a_1(n)\alpha_1^n +\cdots + a_t(n)\alpha_{t}^n,
	\end{equation}
	where $a_1(X), \ldots, a_t(X) \in L[X]$. Assuming that $f(X)$ has no multiple roots, we have that the functions $a_i(X) = a_i$ are all constant for $i=1, \ldots, t=d$. Then the sequence $(U_n)_{n \geq 0}$ is called {\it simple}. Further, the sequence $(U_n)_{n \geq 0}$ is called {\it non-degenerate} if no quotient $\alpha_i/\alpha_j$ for $i\neq j$ is constant.

	\vspace{0.3cm}
	
	At first we will discuss two variants of Diophantine equations related to recurrence sequences in number field setting. The first one is about zero set of a linear recurrence sequence and the second one is about $S$-units in recurrence sequences. The zero set of $(U_n)_{n\geq 0}$ which we denote by $\mathcal{Z} = \mathcal{Z}(\{U_n\})$ is the set of solutions in $n\in \Z$ to the equation $U_n =0$. The Skolem-Mahler-Lech Theorem says that $\mathcal{Z}$ is a union of a finite set and a finite number of infinite arithmetic progressions. Moreover, $\mathcal{Z}$ is finite if the sequence is non-degenerate (see \cite{lech, meh, skolem}). All known proofs of Skolem-Mahler-Lech Theorem are ineffective. This means that we do not know any algorithm that allows us to determine the set  $\mathcal{Z}(\{U_n\})$ of a given linear recurrence $U_n$ defined over a number field of characteristic zero. Furthermore, given linear recurrences $U_n^{(1)}, \ldots, U_{n}^{(k)}$ over a number field $K$, we let 
	\[\mathcal{Z}(U_n^{(1)}, \ldots, U_n^{(k)}) := \{(n_1, \ldots, n_k) \in \N^k \mid U_{n_1}^{(1)}+\ldots +U_{n_k}^{(k)} = 0\}.\] It was conjectured in \cite{Cerlienco1987} that, if $K=\mathbb{Q}$, then the property 
	\[\mathcal{Z}(U_n^{(1)}, \ldots, U_n^{(k)}) \neq \emptyset\] is undecidable for every positive integer $k$ large enough. Furthermore, if $(U_n^{(i)})_{n\geq 0}$ for $i=1, \ldots, k$ be linear recurrences with values in a field of positive characteristic $p$, then the solution set  $\mathcal{Z}(U_n^{(1)}, \ldots, U_n^{(k)})$ is $p$-normal and also effectively computable(see \cite{der15}). 
	
	\vspace{0.3cm}
	
	Next we will consider Diophantine equations combining both $S$-units and recurrence sequences in the number field setting. These type of problems have been investigated by many authors (see for example \cite{bhpr1, bhpr2, brp23, sl, hs, rb25}). For a given recurrence sequence $(U_n)_{n\geq 0}$, several authors have also studied the problem of finding $(n, m, z)$ such that
	\begin{equation}\label{eq2}
		U_{n} + U_{m} = 2^{z}.
	\end{equation}
	Pink and Ziegler \cite[Theorem 1]{Pink2016} investigated a more general Diophantine equation 
	\begin{equation}\label{eq2a}
		U_{n} + U_{m} = wp_{1}^{a_{1}} \cdots p_{s}^{a_{s}}
	\end{equation}
	in nonnegative  unknown integers $n, m, a_{1}, \ldots, a_{s}$, where $(U_{n})_{n \geq 0}$ is a binary non-degenerate recurrence sequence, $p_{1}, \ldots, p_{s}$ are distinct primes and $w$ is a nonzero integer with $p_{i} \nmid w$ for $1\leq i \leq s$. In other words, they studied the situation when the sum $U_n+U_m$ is a so-called $S$-unit, i.e., $U_n+U_m$ has only primes of $S$ in its prime factorization. Further, certain variations of the Diophantine equation \eqref{eq2a} has also been studied (see e.g. \cite{mr2019, mr}) for binary recurrence sequences. 
	
	\vspace{0.3cm}
	
	The problems of estimating the number of zeros appearing in a sequence $(U_n)_{n\geq 0}$ or more generally estimating the number of solutions of  polynomial-exponential equations over a function field are also available. For example, the number of zeros of recurrences in one variable has been estimated \cite{bombieri, Fuchs2005, schmidt}.  In \cite{derksen}, Derksen proved the remarkable result that the zero set of a linear recurrence sequence defined over positive characteristic can always be described in terms of finite automata. The similar analysis for polynomial-exponential equations in several variables has been discussed in \cite{adam, zannier2004}. Furthermore, many finiteness results on the number of solutions of Diophantine equations such as linear equations in $S$-units, and $S$-units in linear recurrence sequences over function field have been studied by many authors \cite{eve1986, Fuchs2021, zannier93}.  In particular, Heintze \cite{heintze} proved that there are only finitely many $S$-units in a linear recurrence sequence of order at least two. 
	
	In \cite{Fuchs2005}, Fuchs and Peth\"{o} considered the Diophantine equation of the type 
	\begin{equation}\label{eq20}
		G_n =H_m,
	\end{equation} where $(G_n)_{n\geq 0}$ and $(H_m)_{m\geq 0}$ are linear recurrence sequences in the function field $L$. 
	In particular, they proved under certain conditions that there exists an effectively computable constant $C$ such that $\max (n, m)\leq C$ for all $(n, m)\in \N^2$ satisfying \eqref{eq20}, unless there are infinitely many solutions coming from a trivial identity. 
	
	In this paper, we are interested in bounding the indices of a linear recurrence sequence, defined over a function field in one variable over an algebraically closed field of characteristic zero, if the sum of more than two terms of this sequence is an $S$-unit. We also study the finiteness result of the Diophantine equation $U_n+V_m+W_\ell = 0$ in $(n,m,\ell)\in \N^3$, where $U_n, V_m,W_\ell$ are simple linear recurrence sequences in the function field.

	\section{Notation and results} 
	We will start with definitions of the discrete valuations on the field $k(x)$, where $x$ is transcendental over $k$.  For $a\in k$ and $g(x)\in k(x)$, we define the valuation $\nu_a$ as follows. Let $g(x) = (x-a)^{\nu_a(g)} p(x)/q(x)$, where $p, q\in k[x]$ and $p(a)q(a)\neq 0$. Further, we write  $\nu_{\infty}(g) = \deg q- \deg p$ for $g(x) = p(x)/q(x)$, where $p, q\in k[x]$. These functions $\nu:k(x) \to \Z$ are all valuations in $k(x)$ up to equivalence. If $\nu_a(g) >0$, then $a$ is called a zero of $g$, and if $\nu_a(g) <0$, then $a$ is called a pole of $g$, where $a\in k\cup \{\infty\}$.  Let $L$ be a finite algebraic extension of function field $k(x)$. All valuations can be extended in at most $[L: k(x)]$ ways to a valuation on $L$ and again in this way one obtains all discrete valuations on $L$.  A valuation on $L$ is called {\em finite} if it extends $\nu_a$ for some $a\in k$ and {\em infinite} if it extends $\nu_{\infty}$. Furthermore, both in $k(x)$ as well as in $L$ the sum formula 
	\begin{equation*}
		\sum_{\nu} \nu(f) = 0
	\end{equation*}
	holds, where the sum is taken over all valuations in the relevant function field. Each valuation in a function field corresponds to a place and vice versa. The places can be thought of as the equivalence classes of valuations (see for more detail \cite{sti1993}).
	
	For a finite set $S$ of valuations on $L$, we call an element of $L$ an $S$-unit, if it has poles and zeros only at places in $S$, or equivalently, the set of $S$-units in $L$ is 
	\begin{equation}
		\mathcal{O}_S^{*} = \{f\in L\mid \nu(f) = 0 \;\; \mbox{for all}\;\; \nu\not \in S\}.
	\end{equation}
In the following result, we investigate a generalization of a result proposed by Heintze (see \cite{heintze}).
	\begin{theorem}\label{thm1}
		Let $L$ be a function field in one variable over an algebraically closed field $k$ of zero characteristic  and $S$ be a finite set of valuations on $L$. Let $(U_n)_{n\geq 0}$ be a simple linear recurrence sequence of order at least two defined over $L$. Assume that $(U_n)_{n\geq 0}$ is non-degenerate,  no $\alpha_i$ lies in $k$, and for all $i\neq j$ the characteristic roots $\alpha_i$ and $\alpha_j$ are multiplicatively independent. Then there exists an effectively computable constant $C$ such that 
		\begin{equation}\label{maineq}
			U_{n_1}+\cdots+U_{n_r}\in \mathcal{O}_S^{*}
		\end{equation}
		for $n_1>\cdots> n_r$ implies $\max(n_1, \ldots, n_r)\leq C$.
	\end{theorem}

	To state our next result, we consider the following. Suppose
	\begin{equation}\label{eq21}
		U_n=\sum_{i=1}^{d_1}a_i\alpha_i^{n},\quad  V_m=\sum_{i=1}^{d_2}b_i\beta_i^{m},\quad  W_\ell=\sum_{i=1}^{d_3}c_i\gamma_i^{\ell}
	\end{equation} are three simple linear recurring sequences in the function field $L$. In particular, we study the Diophantine equation
	\begin{equation}\label{eq22}
		U_n+V_m+W_\ell=0,  
	\end{equation} \text{with }
	\begin{equation}\label{eq11}
		a_i\alpha_i^n+b_j\beta_j^m+c_k\gamma_k^{\ell}\neq 0,
	\end{equation}
	in the unknowns $(n,m,\ell)\in \N^3$ and $1\leq i\leq d_1, 1\leq j\leq d_2, 1\leq k\leq d_3$. Furthermore, we write 
	\begin{equation}
		U_{n}=U_{n}^{(1)}+U_{n}^{(2)}, \quad V_{m}=V_{m}^{(1)}+V_{m}^{(2)},\quad \text{and}\quad W_{\ell}=W_{\ell}^{(1)}+W_{\ell}^{(2)}, 
	\end{equation} where
	\begin{equation*}
	U_{n}^{(1)}=\sum_{j=1}^{t_1}a_j\alpha_j^n,\quad U_n^{(2)}=\sum_{j=t_1+1}^{d_1}a_j\alpha_j^n, \quad (1\leq t_1\leq d_1 ),
	\end{equation*}
	\begin{equation*}
		V_{m}^{(1)}=\sum_{j=1}^{t_2}b_j\beta_j^m,\quad V_m^{(2)}=\sum_{j=t_2+1}^{d_2}b_j\beta_j^m,\quad (1\leq t_2\leq d_2 ), 
	\end{equation*}
	\begin{equation*}
		W_{\ell}^{(1)}=\sum_{j=1}^{t_3}c_j\gamma_j^\ell, \quad W_\ell^{(2)}=\sum_{j=t_3+1}^{d_3}c_j\gamma_j^\ell,\quad (1\leq t_3\leq d_3).
	\end{equation*}	
In our next result, we extend an earlier result by Fuchs and Peth\H{o} (see \cite[Corollary 4.1]{Fuchs2005}).
	\begin{theorem}\label{thm2}
		Suppose  $U_n,V_m$ and $ W_{\ell}$ are simple linear recurring sequences in the function field $L$ and $\alpha_i,\beta_j,\gamma_k\in L^{*}$ for $1\leq i\leq d_1, 1\leq j\leq d_2, 1\leq k\leq d_3 $ with $d_{1},d_2, d_3\geq 2$ are the simple roots of the respective characteristic polynomials. Assume that no $\alpha_i,\beta_i,\gamma_i$ and no ratio $\alpha_i/\alpha_j,\beta_i/\beta_j,\gamma_i/\gamma_j$ lies in $k^{*}$ for each pair of subscripts $i$ and $j$. Then there exists an effectively computable constant $C$ such that all the solutions $(n, m,\ell)$ of \eqref{eq22} which satisfying \eqref{eq11}, we have 
		\begin{equation*}
			\max(n,m,\ell) \leq C
		\end{equation*}
		unless there exist integers $n_0,$ $\ell_0,p_0,q_0,$ $n_1,m_1,p_{1},q_{1},m_2,\ell_2,p_{2},$ and $q_{2}$ such that 
		\begin{equation}\label{thm2eq1}
			U_{n_0+p_0{\ell}}^{(1)}+W_{\ell_0+q_0{\ell}}^{(1)}+U_{n_1+p_{1}{m}}^{(2)}+V_{m_1+q_{1}{m}}^{(1)}=0
		\end{equation}
		or \begin{equation}\label{thm2eq2}
			U_{n_0+p_0{m}}^{(1)}+V_{m_0+q_0{m}}^{(1)}+	U_{n_1+p_1{\ell}}^{(2)}+W_{\ell_1+q_1{\ell}}^{(1)}+V_{m_2+p_{2}{\ell}}^{(2)}+W_{\ell_2+q_{2}{\ell}}^{(2)}=0.
		\end{equation}
		 In this case, we have infinitely many solutions.
	\end{theorem}
	
	\begin{remark}\label{rem3}
		Let $U_{n}=(x-t)^{n}+(x-2t)^{n},V_{m}=(x-t)^{m}+2(x-2t)^{m}$ and $W_{\ell}=(x-2t)^{\ell}-2(x-t)^{\ell}$. Then one can see that $U_{n}+V_{m}+W_{\ell}=0$ has infinitely many solutions if $n=m=\ell$.
Observe  that the minimal vanishing subsums having all the distinct characteristic  roots are 
		\begin{align*}
			(x-t)^{n}+(x-t)^{n}-2(x-t)^{n}=0,\\
			(x-2t)^{n}-2(x-2t)^{n}+(x-2t)^{n}=0,
		\end{align*} and they do not satisfy \eqref{eq11}.
	\end{remark}	
\section{Preliminaries}
	The {\em projective height} $\mathcal{H}$ of $u_1, \ldots, u_n$ of $L/k(x)$, where $n\geq 2$ and not all $u_i$ is zero, is defined by
	\begin{equation}\label{eq1}
		\mathcal{H}(u_1, \ldots, u_n) = -\sum_{\nu} \min (\nu(u_1), \ldots, \nu(u_n)),
	\end{equation}
	where $\nu$ runs over all places of $L/k$. For a single element $f\in L^{*}$, we set
	\begin{equation}\label{eq2}
		\mathcal{H}(f) :=\mathcal{H}(1, f)= -\sum_{\nu} \min (0, \nu(f)),
	\end{equation}
	where the sum is taken over all valuations of $L$; thus for $f\in k(x)$ the height $\mathcal{H}(f)$ is the number of poles of $f$ counted according to multiplicity. We note that if $f\in k[x]$, then $\mathcal{H}(f)  = [L:k(x)]\deg f$.  Also, observe that $\nu(f) \neq 0$ only for a finite number of valuations $\nu$ and hence, by the sum formula $\mathcal{H}(f) = \sum_{\nu} \max (0, \nu(f))$. For $f = 0$, we define $\mathcal{H}(f) = \infty$. This height function satisfies some basic properties which we listed below.
	
	\begin{lemma}\label{lem1}
		Denote as above by $\mathcal{H}$ the projective height on $L/k(x)$. Then for $f, g\in L^{*}$ the following properties hold:
		\begin{enumerate}[label=(\alph*)]
			\item $\mathcal{H}(f) \geq 0$ and $\mathcal{H}(f) = \mathcal{H}(1/f)$,
			\item $\mathcal{H}(f) - \mathcal{H}(g) \leq \mathcal{H}(f+g)\leq \mathcal{H}(f) +\mathcal{H}(g)$,
			\item $\mathcal{H}(f) -\mathcal{H}(g) \leq \mathcal{H}(fg)\leq \mathcal{H}(f) +\mathcal{H}(g)$,
			\item $\mathcal{H}(f^n) = |n|\cdot \mathcal{H}(f)$,
			\item $\mathcal{H}(f)=0\iff f\in k^{*}$,
			\item $\mathcal{H}(P(f)) = \deg P\cdot \mathcal{H}(f)$ for any $P\in k[T]\setminus \{0\}$.         
		\end{enumerate}
	\end{lemma}
	\begin{proof}
		For a detailed proof, you can refer \cite[Lemma 2]{Fuchs2019}.
	\end{proof}
	
	The following theorem is due to Brownawell and Masser \cite{brow1986} and this is an immediate consequence of \cite[Theorem B, Corollary 1]{Fuchs2012}. This result gives an upper bound for the height of $S$-units, which arise as a solution of certain $S$-unit equations.
	
	\begin{proposition}[Brownawell-Masser]\label{browthm}
		Let $F/\C$ be a function field of one variable of genus $\mathfrak{g}$. Moreover, for a finite set $S$ of valuations, let $u_1, \ldots, u_n$ be not all constant $S$-units and 
		\[1+u_1+\cdots +u_n= 0,\] where no proper subsum of the left hand side vanishes. Then we have 
		\[\max_{i=1, \ldots, n} \mathcal{H}(u_i) \leq \binom{n}{2}(|S|+\max(0, 2\mathfrak{g}-2)). \]
	\end{proposition}
	
	The following result is about multiplicatively independent elements.
	
	\begin{lemma}\label{lem2}
		Let $\gamma, \delta\in L/\C$ be multiplicatively independent and $m, n\in \N$. Assume that 
		\[\mathcal{H}\left(\frac{\gamma^n}{\delta^m}\right) \leq M.\]
		Then there exists an effectively computable constant $C$, depending only on $\gamma, \delta, \mathfrak{g}$ and $M$, such that
		\[\max (n, m) \leq  C.\]
	\end{lemma}
To prove Theorem \ref{thm1}, we closely follow the proof of Heintze \cite{heintze}.
	\section{Proof of Theorem \ref{thm1}}
	
	Let $(U_n)_{n\geq 0}$ be as in the theorem. Putting \eqref{eq6} in \eqref{maineq}, we have
	\begin{equation}\label{eq10}
		\sum_{i=1}^d a_i\alpha_i^{n_1}+\cdots+\sum_{i=1}^da_i\alpha_i^{n_r} - u=0,
	\end{equation}
	for some $u\in \mathcal{O}_S^*$. Without loss of generality, we may assume that $\alpha_1, \ldots, \alpha_d$ and $a_1, \ldots, a_d$ are all $S$-units (by increasing the size of the set $S$).
	Let 
	\[C_1:=\binom{rd}{2} (|S|+\max(0, 2\mathfrak{g}-2)).\]  We will apply the Brownawell-Masser inequality in Proposition \ref{browthm} to the polynomial-exponential equation in unknowns $(n_1, \ldots, n_r)\in \N^r$. Therefore, we consider a minimal vanishing subsum of the left hand side of \eqref{eq10}, i.e., no proper sub-subsum of this subsum vanishes. Observe that this subsum contains at least two summand since $a_i\neq 0$ and $\alpha_i\neq 0$ for $1\leq i\leq d$. \\[1pt]
	\textbf{Case I:} \textit{The minimal vanishing subsum contains at least two terms with same exponent}. \\
	
	Consider the minimal vanishing subsum containing two summands $a_i\alpha_i^{n_t}$ and $a_j\alpha_j^{n_t}$ for $i\neq j$ and some $t$ such that $1\leq t\leq r$. Dividing the equation by $a_j\alpha_j^{n_t}$ and using the Proposition \ref{browthm} we get that
	\begin{equation}
		\mathcal{H}\left(\frac{a_i\alpha_i^{n_t}}{a_j\alpha_j^{n_t}}\right)\leq C_1.
	\end{equation} Then \[n_t\cdot\mathcal{H}\left(\frac{\alpha_i}{\alpha_j}\right)=\mathcal{H}\left(\left(\frac{\alpha_i}{\alpha_j}\right)^{n_t}\right)\leq C_1+\mathcal{H}\left(\frac{a_j}{a_i}\right):=C_2.\] Since $(U_n)_{n\geq 0}$ is non-degenerate we obtain that \[n_t\leq \frac{C_2}{\mathcal{H}\left(\frac{\alpha_i}{\alpha_j}\right)}\leq \frac{C_2}{\min_{1\leq \eta\neq \zeta\leq d} \mathcal{H}\left(\frac{\alpha_\eta}{\alpha_\zeta}\right)}:=C_3.\]     \\[1pt]
	\textbf{Case II:} \textit{The minimal vanishing subsum contains at least two terms with different exponents of which one is bounded as in the Case} I.\\

	Suppose the minimal vanishing subsum contains two summands $a_k\alpha_k^{n_t}$ and $a_l\alpha_l^{n_s}$  $( t\neq s)$ such that $n_t\leq C_3$. Dividing this equation by $a_k\alpha_k^{n_t}$ and using Proposition \ref{browthm}, we get \[\mathcal{H}\left(\frac{a_l\alpha_l^{n_s}}{a_k\alpha_k^{n_t}}\right)\leq C_4.\] Therefore, 
	\begin{align*}
		n_s\cdot\mathcal{H}(\alpha_l)=\mathcal{H}(\alpha_l^{n_s})&=\mathcal{H}\left(\frac{a_l\alpha_l^{n_s}}{a_k\alpha_k^{n_t}}\cdot \frac{a_k\alpha_k^{n_t}}{a_l} \right)\\
		&\leq \mathcal{H}\left(\frac{a_l\alpha_l^{n_s}}{a_k\alpha_k^{n_t}} \right)+\mathcal{H}(a_l)+\mathcal{H}(a_k)+n_t\cdot  \mathcal{H}(\alpha_k)\\
		&\leq C_4+ 2\cdot \max_{1\leq i\leq d}\mathcal{H}(a_i)+C_3\cdot  \max_{1\leq i\leq d}\mathcal{H}(\alpha_i):=C_5.
	\end{align*}
	Thus \[n_s\leq \frac{C_5}{\mathcal{H}(\alpha_l)}\leq \frac{C_5}{\min_{1\leq i\leq d}\mathcal{H}(\alpha_i)} :=C_6.\]
	\\[1pt]
	\textbf{Case III:} \textit{The minimal vanishing subsum contains  only some of the exponents which are distinct and all with different characteristic roots as well as not bounded so far.}\\

	Consider the minimal non-vanishing subsum  which contains the terms $a_{i_1}\alpha_{i_1}^{n_{j_1}}$ and $a_{i_k}\alpha_{i_k}^{n_{j_k}}$
	with $1\leq i_1<i_k\leq d$ and $1\leq j_1,j_k\leq r, (j_1\neq j_k)$. Dividing by $a_{i_k}\alpha_{i_k}^{n_{j_k}}$ and using the Proposition \ref{browthm}, we get
	
	\begin{equation}
		\mathcal{H} \left(\frac{a_{i_1}\alpha_{i_1}^{n_{j_1}}}{a_{i_k}\alpha_{i_k}^{n_{j_k}}}\right)\leq C_1. 
	\end{equation}
	Further by applying Lemma \ref{lem1}, we get 
	\begin{equation}\label{eq17}
		\mathcal{H} \left(\frac{\alpha_{i_1}^{n_{j_1}}}{\alpha_{i_k}^{n_{j_k}}}\right)\leq C_{7}. 
	\end{equation}
	Since no $\alpha_i$ lies in $k$, and that for all $i\neq j$ the characteristic roots $\alpha_i$ and $\alpha_j$ are multiplicatively independent, then by Lemma \ref{lem2} we infer that 
	\[\max(n_{j_1}, n_{j_{k}})\leq C_{8}.\] \\[1pt]
	\textbf{Case IV:} \textit{The minimal vanishing subsum contains only some of the exponents which are distinct and all with same characteristic roots as well as none of the exponent is bounded so far.}\\
	
	Suppose that there are $q_1$ number of minimal subsum of such minimal vanishing subsums and let $I = \{n_{j_1}, \ldots, n_{j_k}\}$ with $\# I \leq r-2$ be the set of exponents which is not present in these $q_1$ number of vanishing subsums. Let $n_{i_t}$ be the one of the exponent present in the $t$-th minimal vanishing subsums $(1\leq t\leq q_1)$. Consider one such minimal vanishing subsum that contains the terms $a_i\alpha_i^{n_{i_1}}$ and $a_i\alpha_i^{n_{j}} (j\neq i_1)$ and $n_j\not \in I$.
	Dividing  by $a_i\alpha_i^{n_{i_1}}$ and then applying Proposition \ref{browthm}, we get
	\[(n_{j}-n_{i_1})\cdot\mathcal{H}(\alpha_i) \leq C_9\] and this implies
	\[(n_{j}-n_{i_1}) \leq \frac{C_9}{\min_{1\leq \eta\leq d}\mathcal{H}\left(\alpha_{\eta}\right)}=: C_{10}.\]
	Thus we have $n_{j} = n_{i_1}+\zeta_{j, i_1}^{(1)}$, where $\zeta_{j,i_1}^{(1)}\leq C_{10}^{(1)}$.
	Similarly, from each of the $q_1-1$ minimal vanishing subsums we will get an equation of the form
	\begin{equation}\label{case4a}
		n_j=n_{i_t}+\zeta_{j,i_t}^{(1)}, \quad (1\leq t\leq q_1,\quad n_j\not \in I).
	\end{equation}
	For each of this finitely
	many possibilities we insert the representation \eqref{case4a} in \eqref{eq10} and after rearranging the terms we get
	\begin{equation}\label{eq10a}
		\sum_{n_j\in I} \sum_{\ell=1}^d a_{\ell}\alpha_{\ell}^{n_j}+\sum_{\ell=1}^{d} a_{\ell}^{(1)}\alpha_{\ell}^{n_{i_1}}+\cdots+\sum_{\ell=1}^{d} a_{\ell}^{(1)}\alpha_{\ell}^{n_{i_{q_1}}}-u=0.
	\end{equation}
	Now we enlarge the set $S$ to a set $S_1$ such that $a_{i}^{(1)}$ is an $S_1$-unit for all $i = 1,\ldots,d$. If we can able to find an upper bound for each $n_{i_t}\quad(1\leq t\leq q_1)$, then we are done. 
	
	The equation \eqref{eq10a} is one similar to the main equation \eqref{eq10} with less number of summands than \eqref{eq10}. Now we will consider different cases at which \eqref{eq10a} has a vanishing subsum (note that, here we will not consider the situation where the minimal vanishing subsums of \eqref{eq10a} which are of the form in Case I, II, and III, as we can deal those in a similar way). So from now onwards we will consider 
	\begin{equation}\label{eq10b}
\sum_{\ell=1}^{d} a_{\ell}^{(1)}\alpha_{\ell}^{n_{i_1}}+\cdots+\sum_{\ell=1}^{d} a_{\ell}^{(1)}\alpha_{\ell}^{n_{i_{q_1}}}-u=0.
	\end{equation}
	
	If it has $q_2$ number of  minimal vanishing subsum with different exponents that are not bounded so far with same characteristic roots occur at least twice, then we do the  same process which we done above and again get an equation of the form \eqref{eq10b} with less number of summands than the number of  summands in \eqref{eq10b}. Our aim is to rewrite the equation \eqref{eq10b} to an equation with exactly one exponent occur, say $n_{j_s}(1\leq j_s\leq r)$.  Since the number of terms in \eqref{eq10} is finite, after a finite number of step (say, $m$-th step), we will get equations of the form 
	\begin{equation}\label{eq10c}
		\sum_{l=1}^{d} a_{l}^{(m)}\alpha_{l}^{n_{j_1}}+\cdots+\sum_{l=1}^{d} a_{l}^{(m)}\alpha_l^{n_{j_{s}}}-u=0 \quad \text{and} \quad 	n_{j_i}=n_{j_s}+\zeta_{j_i, j_s}^{(m)} 
	\end{equation}
	for all $1\leq i\leq s-1$, where $\eta_{j_i, j_s}\zeta_{j_i, j_s}^{(m)} \leq C_{10}^{(m)}$ and it has no minimal vanishing subsum.
Now substituting $n_{j_i}=n_{j_s}+\zeta_{j_i, j_s}^{(m)} $ for $1\leq i\leq s-1$ in \eqref{eq10c} and with some further  simplifications, we get
	\begin{equation*}
		\sum_{i=1}^{d}a_{i}^{(m+1)}\alpha_{i}^{n_{j_s}}-u=0,\quad \text{for some}\quad 1\leq j_s\leq r
	\end{equation*}
	and \[a_i^{(m+1)}=a_i^{(m)}(\alpha_i^{\zeta_{j_1, j_s}^{(m)}}+\alpha_i^{\zeta_{j_2, j_s}^{(m)}}+\cdots+1).\]
	
	Now by \cite[Theorem 1]{heintze}, we get $n_{j_s}\leq C_{11}$ and this implies $n_{j_i} \leq C_{11}+C_{10}^{(m)}=:C_{12}$ for $1\leq i\leq s-1$. Hence, all the exponents in \eqref{eq10c} is bounded. Next we consider  the similar equation like \eqref{eq10c} at the $(m-1)$-th step which has $dt(t>s)$ number of summands with exponents $n_{i_l}$ such that $n_{i_l}=n_{j_k}+\zeta^{(m-1)}_{i_l,j_k}$, where $ i_l\in\{1,\ldots, r\} $ and $1\leq k\leq s$. Proceeding as above and since the number of exponents present in \eqref{eq10} is finite, we can bound all the exponents $n_i$ for all $i=1,\ldots, r$. This completes the proof of Theorem \ref{thm1}.
	\qed

	\section{Proof of Theorem \ref{thm2}}
We use the following result due to Fuchs and Peth\H{o} \cite{Fuchs2005} several times in the proof of Theorem \ref{thm2}.
\begin{lemma}\label{lemfuchs}
		Suppose  $U_n$ and $V_m$ are simple linear recurring sequences in the function field $L$ and $\alpha_i,\beta_j\in L^{*}$ for $1\leq i\leq d_1, 1\leq j\leq d_2$ with $d_{1}, d_2\geq 2$ are the simple roots of the respective characteristic polynomials. Assume that no $\alpha_i,\beta_i$ and no ratio $\alpha_i/\alpha_j$ or $\beta_i/\beta_j, i\neq j$ lies in $k^{*}$. Then there exists an effectively computable constant $C$ such that all the solutions $(n, m)\in \N^2$ with $U_n= V_m$, we have 
		\begin{equation*}
			\max(n, m) \leq C
		\end{equation*}
		unless there are integers $n_0, m_0, r, s$ with $rs\neq 0$ such that the pairs $(a_i\alpha_i^{n_0}, \alpha_i^r)$ coincide in some order with the pairs $(b_i\beta_i^{m_0}, \beta_i^s)$. In this case, we have infinitely many solutions.
\end{lemma}	
	
	\begin{proof}
		See \cite[Corollary 4.1]{Fuchs2005}.
	\end{proof}
We also need the following two results to prove Theorem \ref{thm2}. 
\begin{proposition}\label{thm2prop1}
Assume all the hypothesis in Theorem \ref{thm2} and $d_1\geq d_2+d_3$. Suppose that the minimal vanishing subsum contains none of the two summands with the  same exponent. Then \eqref{eq22} has finitely many solutions unless there exist integers $n_0,$ $\ell_0,p_0,q_0,$ $n_1,m_1,p_{1}$ and $q_{1}$ satisfying \eqref{thm2eq1}.
\end{proposition}	

\begin{proof}
Consider minimal vanishing subsum of \eqref{eq22}. Without loss of generality we can assume that $d_2\geq d_3$. Given that $d_1\geq d_2+d_3$. If $d_1>d_2+d_3$, then there will be no terms with exponent $m$ or $\ell$ to pair with the $d_1-(d_2+d_3)$ terms with exponent $n$. So, we must have $d_1=d_2+d_3$ and rewriting the indices if necessary we get
	\begin{align*}\label{eq22ali}
		\begin{split}
			&a_1\alpha_1^{n} + c_1\gamma_1^\ell = 0, \ldots, a_{d_3}\alpha_{d_3}^{n} +c_{d_3}\gamma_{d_3}^{\ell}   = 0,\\
			&a_{d_3+1}\alpha_{d_3+1}^{n} + b_{1}\beta_{1}^{m} = 0,\ldots, a_{d_1}\alpha_{d_1}^{n} + b_{d_2}\beta_{d_2}^{m} = 0.\\
		\end{split}
	\end{align*}
	Thus, we have equations of the form
	\begin{equation}\label{case1a}
		a_i\alpha_i^n=-c_j\gamma_j^\ell \quad(1\leq i,j\leq d_3)
	\end{equation} 
	and
	\begin{equation}\label{case1b}
	a_s\alpha_s^n	=-b_t\beta_t^m \quad (d_3+1\leq s\leq d_1, 1\leq t\leq d_2)
	\end{equation}
	The equation \eqref{case1a}  (resp. \eqref{case1b}) can only hold for one pair $(n_0,\ell_0)\in \N^{2}$ (resp. $(n_1,m_1)\in \N^{2}$) unless the set of zeros and poles of $\alpha_i$ and $\gamma_j$ (resp. $\alpha_s$ and $\beta_t$) coincide.  Let $\nu$ and $\mu$ be one of these zeros and poles which coincide in the above two equations respectively. Then \eqref{case1a} implies that 
	\begin{equation}\label{case1c}
		n=p_0\ell+q_0 
	\end{equation} for the integers $	p_0=\frac{\nu(\gamma_j)}{\nu(\alpha_i)}, \quad q_0=\frac{\nu(c_j)-\nu(a_i)}{\nu(\alpha_i)}.$ Also, \eqref{case1b} implies that
	\begin{equation}\label{case1d}
		n=p_{1}m+q_{1}	
	\end{equation} for integers $
	p_{1}=\frac{\mu(\beta_t)}{\mu(\alpha_s)}, \quad q_{1}=\frac{\mu(b_t)-\mu(a_s)}{\mu(\alpha_s)}.$
	Putting \eqref{case1c} in \eqref{case1a} and \eqref{case1d} in \eqref{case1b} we get
	\[\alpha_i^{q_0}\frac{a_i}{c_j}=\left(\frac{\gamma_j}{\alpha_i^{p_0}}\right)^{\ell}\quad \text{and}\quad \alpha_i^{q_1}\frac{a_s}{b_t}=\left(\frac{\beta_t}{\alpha_s^{p_{1}}}\right)^{m}.\]
	This equation holds for one $\ell=\ell_0$ (resp. $m=m_0$) only, unless $\gamma_j/\alpha_i^{p_0}$ (resp. $\beta_t/\alpha_s^{p_{1}}$) is a root of unity and $a_i/c_j$ (resp. $a_s/b_t$) is also a root of unity. Then the equation holds for all $\ell$ (and $m$ respectively) in an arithmetic progression.
	
	Thus, we conclude that the equations hold just for one $(n_0,\ell_0)$ and $(n_1,m_1)$ in $\N^{2}$, unless we have $d_1=d_2+d_3$ and there exists integers $n_0,\ell_0,p_0,q_0,n_1,m_1,p_1,q_1$ and bijections $\rho$ from $\{1,\ldots, d_3\}$ to itself and $\delta$ from $\{d_3+1,\ldots,d_1\}$ to $\{1,\ldots,d_2\}$ such that \eqref{case1a} and \eqref{case1b} holds, $$\frac{\alpha_i^{p_0}}{\gamma_{\rho(i)}^{q_0}}=u_i,\quad \frac{\alpha_j^{p_{1}}}{\beta_{\delta(j)}^{q_{1}}}=v_j, $$ where $u_i$ and $v_j$ are root of unity. Therefore,
	\begin{equation}\label{eq28}
		\sum_{i=1}^{d_3}\left( a_i\alpha_i^{n_0+p_0\ell}+c_{\rho(i)}\gamma_{\rho(i)}^{\ell_0+q_0\ell}\right)+\sum_{j=d_3+1}^{d_1}\left( a_j\alpha_j^{n_1+p_{1}m}+b_{\delta(j)}\beta_{\delta(j)}^{m_1+q_{1}m}\right)=0,
	\end{equation}
i.e.,
	\begin{equation*}
		\sum_{i=1}^{d_3}\left( a_i\alpha_i^{n_0}u_i^\ell+c_{\rho(i)}\gamma_{\rho(i)}^{\ell_0}\right)\gamma_{\rho(i)}^{q_0\ell}+\sum_{j=d_3+1}^{d_1}\left( a_j\alpha_j^{n_1}v_j^{m}+b_{\delta(j)}\beta_{\delta(j)}^{m_1}\right)\beta_{\delta(j)}^{q_{1}m}=0
	\end{equation*}
	The above equation has infinitely many solutions if all the coefficient vanishes, that is, $-a_i\alpha_i^{n_0}/c_j\gamma_j^{\ell_0}$ and $-a_s\alpha_s^{n_1}/b_t\beta_t^{m_1}$ are root of unity. Otherwise by setting		$$	A_i:= a_i\alpha_i^{n_0}u_i^\ell+c_{\rho(i)}\gamma_{\rho(i)}^{\ell_0}, \quad\gamma_i^{'}:=\gamma_{\rho(i)}^{q_0} ,\quad 1\leq i\leq d_3$$ and  $$B_j:=a_j\alpha_j^{n_1}v_j^{m}+b_{\delta(j)}\beta_{\delta(j)}^{m_1},\quad \beta_j^{'}:=\beta_{\delta(j)}^{q_{1}},\quad d_3+1\leq j\leq d_1,$$ we get
		\begin{equation}
			\sum_{i=1}^{d_3}A_i\gamma_i^{'^{\ell}}+\sum_{j=d_3+1}^{d_1}B_j\beta_j^{'m}=0.
		\end{equation}
		By Lemma \ref{lemfuchs}, there exists an effectively computable constant $C$ such that $\max\{m,\ell\}\leq C$ unless there are integers $m_0,\ell_0,p_0,q_0$ such that the pairs $(A_i\gamma_i^{'\ell_0}, \gamma_i^{'p_0})$ coincide in some order with the pairs $(-B_j\beta_j^{'m_0},\beta_j^{'q_0}).$ This completes the proof.
\end{proof}

\begin{proposition}\label{thm2prop2}
Assume all the hypothesis in Theorem \ref{thm2} and $d_1< d_2+d_3$. Suppose that the minimal vanishing subsum contains none of the two summands with the  same exponent. Then \eqref{eq22} has finitely many solutions unless there exist integers  $n_0,$ $\ell_0,p_0,q_0,$ $n_1,m_1,p_{1},q_{1},m_2,$ $\ell_2,p_{2},$ and $q_{2}$ satisfying \eqref{thm2eq2}.
\end{proposition}

\begin{proof}
Since $d_1<d_2+d_3$, we can write the subsums as follows:
			\begin{align*}
			\begin{split}
				&a_1\alpha_1^n+b_1\beta
				_1^m=0, \ldots, a_t\alpha_t^n+b_t\beta_t^m=0,\\
				&a_{t+1}\alpha_{t+1}^n+c_1\gamma_1^\ell=0,\ldots, a_{d_1}\alpha_{d_1}^n+c_s\gamma_s^{\ell}=0,\\
			&b_{t+1}\beta_{t+1}^m+c_{s+1}\gamma_{s+1}^\ell=0,\ldots, b_{d_2}\beta_{d_2}^m+c_{d_3}\gamma_{d_3}^{\ell}=0.\\
			\end{split}
		\end{align*}
		Thus, we have equations of the form
		\begin{equation}\label{case2a}
			a_i\alpha_i^{n}=-b_j\beta_j^m, \quad(1\leq i,j\leq t)
		\end{equation}
		\begin{equation}\label{case2b}
			a_x\alpha_x^n=-c_y\gamma_y^{\ell},\quad (t+1\leq x\leq d_1; 1\leq y\leq s),
		\end{equation}
		\begin{equation}\label{case2c}
			b_u\beta_u^m=-c_v\gamma_v^\ell, \quad (t+1\leq u\leq d_2; s+1\leq v\leq d_3),
		\end{equation}
		The equation \eqref{case2a} can only hold for one pair $(n_0,m_0)\in \N^{2}$ unless the set of zeros and poles of $\alpha_i$ and $\beta_j$ coincide. Similarly  the equation \eqref{case2b} and \eqref{case2c} can only hold for one pair $(n_1,\ell_1)\in \N^{2}$ and $(m_2,\ell_2)\in \N^{2}$  unless the set of zeros and poles of $\alpha_x$ and $\gamma_y$ coincide and the set of zeros and poles of $\beta_u$ and $\gamma_v$  coincide respectively. Let $\nu, \mu$ and $\tau$ be one of these zeros and poles which coincide in the above three equations respectively. Then \eqref{case2a}, \eqref{case2b} and   \eqref{case2c}  implies that 
		\begin{equation}\label{case2d}
		n=p_0m+q_0, \quad n= p_1\ell+q_1\quad \mbox{and}\quad 	m=p_2\ell+q_2
		\end{equation}
		for the integers 
		\begin{align}
		\begin{split}\label{case2e}
		&p_0=\frac{\nu(\beta_j)}{\nu(\alpha_i)}, q_0=\frac{\nu(b_j)-\nu(a_i)}{\nu(\alpha_i)}, p_1=\frac{\mu(\gamma_y)}{\mu(\alpha_x)},\\
		& q_1=\frac{\mu(c_y)-\mu(a_x)}{\mu(\alpha_x)}, p_2=\frac{\tau(\gamma_v)}{\tau(\beta_u)}, q_2=\frac{\tau(c_v)-\tau(b_u)}{\tau(\beta_u)}.
		\end{split}
		\end{align} Thus from \eqref{case2d}, \eqref{case2a}, \eqref{case2b} and \eqref{case2c} we get
		\begin{equation}
			\alpha_i^{q_0}\frac{a_i}{b_j}=\left(\frac{\beta_j}{\alpha_i^{p_0}}\right)^{m}, \quad 	\alpha_x^{q_1}\frac{a_x}{c_y}=\left(\frac{\gamma_y}{\alpha_x^{p_1}}\right)^{\ell}, \quad 	\beta_{u}^{q_2}\frac{b_u}{c_v}=\left(\frac{\gamma_v}{\beta_u^{p_2}}\right)^{m}.
		\end{equation}
		The above equations can hold for one $m=m_0, \ell=\ell_1$ and $m=m_2$ respectively only, unless $a_i/b_j$, $\beta_j/\alpha_i^{p_0}$, $ a_x/c_y$, $\gamma_y/\alpha_x^{p_1}$, $b_u/c_v$ and $\gamma_v/\beta_u^{p_2}$ are root of unity. Then the equations hold for all $m,\ell$ in arithmetic progression.
		
		Thus, we conclude that the equations hold for one $(n_0,m_0), (n_1,\ell_1)$ and $(m_2,\ell_2)$ unless we have $d_1<d_2+d_3$ and there exists integers $n_0,m_0,p_0,q_0,n_1,\ell_1,p_1,q_1,m_2,\ell_2,p_2,q_2$ and bijection  $\rho$ from $\{1,\ldots,t\}$ to itself, $ \delta$ from $ \{t+1,\ldots, d_1\}$ to $\{1,\ldots, s\}$ and $\sigma$ from $\{d_1+1,\ldots, d_2+d_3-d_1 \}$ to $\{s+1,\ldots, d_3\}$ such that \eqref{case2a}, \eqref{case2b} and \eqref{case2c} hold respectively, $$\frac{\alpha_i^{p_0}}{ \beta_j^{q_0}}=\omega_i,\quad \frac{\alpha_x^{p_{1}}}{\gamma_y^{q_{1}}}=\lambda_x,\quad \frac{\beta_u^{p_{2}}}{\gamma_v^{q_{2}}}= \epsilon_u,$$ where $\omega_i,\lambda_x, \epsilon_u$ are roots of unity. So,
		\begin{align}\label{eq37}
			\begin{split}
			\sum_{i=1}^{t}\left(a_i\alpha_i^{n_0+p_0m}+b_{\rho(i)}\beta_{\rho(i)}^{m_0+q_0m}\right)&+\sum_{x=t+1}^{d_1}\left(a_x\alpha_x^{n_1+p_1\ell}+c_{\delta(x)}\gamma_{\delta(x)}^{\ell_1+q_1\ell}\right)\\
			&+\sum_{u=d_1+1}^{\frac{d_2+d_3-d_1}{2}}\left(b_u\beta_u^{m_2+p_2\ell}+c_{\sigma(u)}\gamma_{\sigma(u)}^{\ell_2+q_2\ell}\right)=0,
			\end{split}
		\end{align}
		that is,
		\begin{align*}
			\sum_{i=1}^{t}\left(a_i\alpha_i^{n_0}\omega_i^{m}+b_{\rho(i)}\beta_{\rho(i)}^{m_0}\right)\beta_{\rho(i)}^{q_0m}&+\sum_{x=t+1}^{d_1}\left(a_x\alpha_x^{n_1}\lambda_x^\ell+c_{\delta(x)}\gamma_{\delta(x)}^{\ell_1}\right)\gamma_{\delta(x)}^{q_{1}\ell}\\
			&+\sum_{u=d_1+1}^{\frac{d_2+d_3-d_1}{2}}\left(b_u\beta_u^{m_2}\epsilon_u^{\ell}+c_{\sigma(u)}\gamma_{\sigma(u)}^{\ell_2}\right)\gamma_{\sigma(u)}^{q_{2}\ell}=0.\\
		\end{align*}
	The above equation has infinitely many solution if all the coefficient vanishes, i.e., $-a_i\alpha_i^{n_0}/b_j\beta_j^{m_0}$, $-a_x\alpha_x^{n_1}/c_y\gamma_y^{\ell_1}$ and $-b_u\beta_u^{m_2}/c_v\gamma_v^{\ell_2}$ are root of unity.  Otherwise by setting $$A_i=a_i\alpha_i^{n_0}\omega_i^{m}+b_{\rho(i)}\beta_{\rho(i)}^{m_0},\quad \beta_i^{'}=\beta_{\rho(i)}^{q_0},\quad 1\leq i\leq t$$ and $$B_j=\left\{ \begin{array}{cl}
		a_x\alpha_x^{n_1}\lambda_x^\ell+c_{\delta(x)}\gamma_{\delta(x)}^{\ell_1},& t+1\leq x\leq d_1\\
		&\\
		b_u\beta_u^{m_2}\epsilon_u^{\ell}+c_{\sigma(u)}\gamma_{\sigma(u)}^{\ell_2}, & d_1+1\leq u\leq \frac{d_2+d_3-d_1}{2},
	\end{array}\right.$$ \\
	$$\gamma_j^{'}=\left\{ \begin{array}{cl}
	\gamma_{\delta(x)}^{q_{1}}, & t+1\leq x \leq d_1\\
	&\\
	\gamma_{\sigma(u)}^{q_{2}}, & d_1+1\leq u \leq \frac{d_2+d_3-d_1}{2},
	\end{array}\right. $$ we get
	\begin{equation*}
		\sum_{i=1}^{t}A_i\beta_i^{'m}+\sum_{j=t+1}^{\frac{d_2+d_3-d_1}{2}}B_j\gamma_j^{'\ell}=0.
	\end{equation*}
	By Lemma \ref{lemfuchs}, there exists an effectively computable constant $C$ such that $\max\{m,\ell\}\leq C$ unless there are integers $m_0,\ell_0,p_0,q_0$ such that the pairs $(A_i\beta_i^{'m_0},\beta_i^{'p_0})$ coincide in some order with the pairs $(-B_j\gamma_j^{'\ell_0},\gamma_j^{'q_0}).$ 
\end{proof}

\subsection{Proof of Theorem \ref{thm2}:}	Suppose $U_n,V_m$ and $ W_\ell$  are simple linear recurring sequences as in the theorem. Putting \eqref{eq21} in \eqref{eq22}, we infer that
	\begin{equation}\label{eq22a}
		a_1\alpha_1^n+\cdots+ a_{d_1}\alpha_{d_1}^n+b_1\beta_1^m+\cdots+b_{d_2}\beta_{d_2}^m+c_1\gamma_1^\ell+\cdots+c_{d_3}\gamma_{d_3}^\ell=0.
	\end{equation}
	Let $S$ be a finite set of valuations such that all $a_i,b_j,c_k, \alpha_i,\beta_j,$ and $\gamma_k$  are $S$-units for $1\leq i\leq d_1,1\leq j\leq d_2, 1\leq k\leq d_3$. Let 
	\[C_{13}:=\binom{d_1+d_2+d_3}{2} (|S|+\max(0, 2\mathfrak{g}-2)).\] 
	We will apply the Brownawell-Masser inequality in Proposition \ref{browthm} to the polynomial-exponential equation \eqref{eq22a} in unknowns $(n,m,\ell)\in \N^3$. Therefore, we consider a minimal vanishing subsum of the left hand side of \eqref{eq22a}, i.e., no proper sub-subsum of this subsum vanishes. Assuming this set of solutions is non-empty and since each summands in \eqref{eq22a} is nonzero, we can conclude that each minimal vanishing subsum contains at least two elements. 
	
	We consider a minimal vanishing subsum of the left hand side of \eqref{eq22a} containing two distinct elements say, $a_i\alpha_i^n,a_j\alpha_j^n(i\neq j)$.
	After dividing by $a_j\alpha_j^n$ and then applying Proposition \ref{browthm} to this subsum we get
	\[\mathcal{H} \left(\frac{a_i\alpha_{i}^{n}}{a_{j}\alpha_{j}^{n}}\right) \leq C_{13}.\]
	Then
	\begin{align*}
		n\cdot \mathcal{H}\left(\frac{\alpha_{i}}{\alpha_{j}}\right) & = \mathcal{H}\left(\frac{\alpha_i^{n}}{\alpha_{j}^{n}}\right) = \mathcal{H}\left(\frac{a_i\alpha_{i}^{n}}{a_{j}\alpha_{j}^{n}}\cdot \frac{a_j}{a_i}\right) \\
		&\leq\mathcal{H}\left(\frac{a_i\alpha_{i}^{n}}{a_{j}\alpha_{j}^{n}}\right) + \mathcal{H}\left(\frac{a_{j}}{a_{i}}\right)\\
		& \leq C_{13}+ \max_{1\leq {\xi}\neq \eta\leq d_1}\mathcal{H}\left(\frac{a_{\eta}}{a_{\xi}}\right)=:C_{14}.
	\end{align*}
	Since the sequence $(U_n)_{n\geq 0}$ is non-degenerate, we get
	\[n\leq \frac{C_{14}}{ \mathcal{H}\left(\frac{\alpha_{i}}{\alpha_{j}}\right) } \leq \frac{C_{14}}{\min_{\eta\neq \xi} \mathcal{H}\left(\frac{\alpha_{\xi}}{\alpha_{\eta}}\right) }=:C_{15}.\]
	It may happen that this minimal vanishing subsum contain at least one element of the form $b_k\beta_k^m$ or $c_k\gamma_k^\ell$ for any $k$. Then 
	\begin{align*}
		m\cdot \mathcal{H}&\left(\beta_{k}\right)  = \mathcal{H}\left(\beta_k^m\right) \leq \mathcal{H}(b_k\beta_k^m)+ \mathcal{H}(b_k)\\
		& \leq \mathcal{H}\left(\frac{b_k\beta_k^m}{a_j\alpha_j^n} \cdot  a_i\alpha_i^n\right) + \mathcal{H}\left(b_k\right)\\
		& \leq C_{13}+ 2\cdot \max_{1\leq {j}\leq d_1; 1\leq k\leq d_2}\left\{\mathcal{H}\left(a_j\right),\mathcal{H}\left(b_k\right)\right\}+n\cdot \mathcal{H}(\alpha_{i})\\
		& \leq C_{13}+ 2\cdot \max_{1\leq {j}\leq d_1; 1\leq k\leq d_2}\left\{\mathcal{H}\left(a_j\right),\mathcal{H}\left(b_k\right)\right\}+ C_{15}\cdot \max_{1\leq {\eta}\leq d_1}\mathcal{H}\left(\alpha_{\eta}\right)=:C_{16}.
	\end{align*}Then
	\[m\leq \frac{C_{16}}{ \mathcal{H}\left(\beta_k\right)} \leq \frac{C_{16}}{\min_{1\leq {\eta}\leq d_2}\mathcal{H}\left(\beta_{\eta}\right)}=:C_{17}.\]
	Similarly we can bound $\ell$ also.  We may now assume  that the minimal vanishing subsum contains none of the two summands with the  same exponent. i.e., each minimal non-vanishing subsum contains precisely two elements or three elements. Since we assumed the given recurrence sequences satisfy \eqref{eq11}, we consider only the case when each minimal vanishing subsum containing precisely two terms.
	Here we observe that $d_1+d_2+d_3$ is even and there will be at least $\max\{d_1,d_2,d_3\}$ number of minimal vanishing subsums of \eqref{eq22a}. Without loss of generality we assume that $d_1\geq d_2\geq d_3$. Now we consider two cases:
	
	 If $d_1\geq d_2+d_3$, then by Proposition \ref{thm2prop1},  \eqref{eq22} has finitely many solutions in $(n, m, \ell)$  unless there exist integers $n_0, \ell_0, p_0, q_0, n_1, m_1, p_{1}$ and $q_{1}$ satisfying \eqref{thm2eq1}. In this case we have infinitely many solutions.

	If $d_1< d_2+d_3$, then by Proposition \ref{thm2prop2},  \eqref{eq22} has finitely many solutions in $(n, m, \ell)$  unless there exist integers \[n_0, \ell_0,p_0,q_0, n_1,m_1,p_{1},q_{1},m_2, \ell_2,p_{2}, q_{2}\] satisfying \eqref{thm2eq2}. In this case also we have infinitely many solutions. This completes the proof of Theorem \ref{thm2}. \qed

If each minimal vanishing subsum of \eqref{eq22} contains precisely three terms with different exponents then we have $d_1=d_2=d_3=d$. Thus, we will get equations of the form 	
	\begin{equation}\label{eq40}
		a_i\alpha_i^n+b_i\beta_i^m+c_i\gamma_i^\ell=0, \quad 1\leq i\leq d.
	\end{equation} In the following remark, we show that \eqref{eq22} has finitely many solutions when the characteristic roots $\alpha_i,\beta_i$ and $\gamma_i$ are pairwise multiplicatively independent for all $i=1,\ldots, d$ and satisfy \eqref{eq40}.
	
		\begin{remark}\label{rem1}
			Suppose that the characteristic roots $\alpha_i,\beta_i$ and $\gamma_i$ are pairwise multiplicatively independent for all $i=1,\ldots,d$. Then after dividing the equation \eqref{eq40} by $c_i\gamma_i^{\ell}$ and using the Proposition \ref{browthm}, we get
			\begin{equation*}
				\mathcal{H}\left(\frac{a_i\alpha_i^n}{c_i\gamma_i^\ell}\right)\leq C_{18} \quad \text{and}\quad \mathcal{H}\left(\frac{b_i\beta_i^m}{c_i\gamma_i^\ell}\right)\leq C_{18}.
			\end{equation*}
			This implies, 
			\begin{equation*}
				\mathcal{H}\left(\frac{\alpha_i^n}{\gamma_i^\ell}\right)\leq C_{19} \quad \text{and}\quad \mathcal{H}\left(\frac{\beta_i^m}{\gamma_i^\ell}\right)\leq C_{20}.
			\end{equation*}
			 By Lemma \ref{lem2}, we can bound $n, m$ and $\ell$. Thus, we can find finite number of solutions for \eqref{eq22} if all the roots are pairwise multiplicatively independent.
		\end{remark}
		
If the characteristic roots $\alpha_i,\beta_i$ and $\gamma_i$ are pairwise multiplicatively dependent, then one can see that there are infinitely many solutions of \eqref{eq22} satisfying \eqref{eq40} in an arithmetic progression.
	\begin{remark}\label{rem2}
		 Suppose that the characteristic roots $\alpha_i,\beta_i$ and $\gamma_i$ are pairwise multiplicatively dependent. Then there exist $r_i, s_i, t_i, k_i\in \Z$ such that $\alpha_i^{s_i}=\beta_i^{r_i}$ and $\alpha_i^{t_i}=\gamma_i^{k_i}$. Without loss of generality, assume that $r_i<s_i $ and $t_i<k_i$. Then we get 
		\begin{equation*}
			\left(\frac{\beta_i}{\alpha_i}\right)^{r_i}=	\left(\frac{\gamma_i}{\alpha_i}\right)^{t_i}=1.
		\end{equation*}
		Suppose that  $(n_0,m_0,\ell_0)\in \N^{3}$ is a particular solution of \eqref{eq40}. Putting $h_i=lcm(r_i, t_i), 1\leq i\leq d$ and $h=lcm(h_1,\ldots, h_d)$, we have 
			\begin{align*}
			&U_{n_0+hu}+V_{m_0+hu}+W_{\ell_0+hu}\\&= \sum_{i=1}^{d}\left(a_i\alpha_i^{n_0+hu}+b_i\beta_i^{m_0+hu}+c_i\gamma_i^{\ell_0+hu} \right)\\
			&=\alpha_i^{hu}\left[a_i\alpha_i^{n_0}+ b_i\beta_i^{m_0} \left(\frac{\beta_i}{\alpha_i}\right)^{hu}+c_i\gamma_i^{\ell_0}\left(\frac{\gamma_i}{\alpha_i}\right)^{hu} \right]=0.
		\end{align*}for any $u\in \N$. So we can conclude that the set \[\{(n_0+hu, m_0+hu, \ell_0+hu): u\in \N\}\] form a solution for \eqref{eq22}. 
	
		\end{remark}

		{\bf Acknowledgment:} D.N. and S.S.R. is supported by grant from Science and Engineering Research Board (SERB)(File No.:CRG/2022/000268) and S.S.R. is also supported by National Board for Higher Mathematics (NBHM), Sanction Order No: 14053.

	\end{document}